\documentclass [11pt,twoside,a4paper]{article}
\usepackage{amsfonts}
\usepackage{amsthm}
\usepackage{amsmath}
\usepackage{amscd}
\usepackage{amsbsy}            %
\usepackage{psfrag}            %
\usepackage{epsf}              %
\usepackage{graphicx}          %
\usepackage{makeidx}           %
\usepackage{color}             %
\usepackage{fancyhdr}

\newtheorem{thm}{Theorem}[section]
\newtheorem{exm}[thm]{Example}

\newtheorem{lem}[thm]{Lemma}

\newtheorem{prob}[thm]{Problem}
\newtheorem{prop}[thm]{Proposition}
\newtheorem{cor}[thm]{Corollary}

\newcommand{\rank}{\texttt{rank}}
\newcommand{\abs}[1]{\left\vert#1\right\vert}
\newcommand{\norm}[1]{\parallel\! #1\! \parallel}

\newcommand{\set}[1]{\left\{#1\right\}}
\newcommand{\tA}{\tilde{\mathcal{\A}}}
\newcommand{\cprk}{\textit{cprank}}
\newcommand{\Gram}{\texttt{Gram}}
\newcommand{\supp}{\texttt{supp}}
\newcommand{\al}{\alpha}
\newcommand{\allone}{\ell}
\newcommand{\be}{\beta}

\newcommand{\Gam}{\gamma}

\newcommand{\la}{\lambda}
\newcommand{\La}{\Lambda}

\newcommand{\eps}{\epsilon}
\newcommand{\si}{\sigma}
\newcommand{\pdg}{multi-hypergraph}

\newcommand{\A}{\mathcal{A}}
\newcommand{\B}{\mathcal{B}}
\newcommand{\C}{\mathcal{C}}
\newcommand{\D}{\mathcal{D}}

\newcommand{\G}{\mathcal{G}}
\newcommand{\cI}{\mathcal{I}}
\newcommand{\J}{\mathcal{J}}

\newcommand{\bcp}{\set{0,1}\!-\!CP }

\newcommand{\cp}{\mathcal{C\!P}}

\newcommand{\cP}{\mathcal{P}}

\newcommand{\T}{\mathcal{T}}
\newcommand{\bu}{\textbf{u}}

\newcommand{\bx}{\textbf{x}}

\newcommand{\bfe}{\textbf{e}}

\def\CP{\mathbb{C\!P}}
\def\capEb{\mathbb{E}}
\def\Fb{\mathbb{F}}

\def\RR{\mathbb R}
\def\RPlus{RR_{+}}
\def\ZZ{\mathbb Z}
\newcommand{\Zplus}{\ZZ_{+}}

\def\SS{{\mathbb S}}

\newcommand{\SF}{\mathbb{S\!F}}
\newcommand{\ST}{\mathbb{S\!T}}

\def\dsum{\displaystyle\sum}

\newcommand{\mnrt}{$m$th order $n$-dimensional real tensor }
\newcommand{\mnrts}{$m$th order $n$-dimensional real tensors }
\newcommand{\mnst}{$m$th order $n$-dimensional symmetric tensor }
\newcommand{\mnsts}{$m$th order $n$-dimensional symmetric tensors }

\newcommand{\beq}{\begin{equation}}
\newcommand{\eeq}{\end{equation}}
\newcommand{\bey}{\begin{eqnarray}}
\newcommand{\eey}{\end{eqnarray}}
\newcommand{\beyy}{\begin{eqnarray*}}
\newcommand{\eeyy}{\end{eqnarray*}}

\title{On $\set{0,1}$ CP Tensors and CP Multi-hypergraphs}
\author{Changqing Xu\thanks{School of Mathematics and Physics, Suzhou University of Science and Technology, Suzhou, P.R. China. Email: cqxurichard@mail.usts.edu.cn.},
Zhibing Chen\thanks{College of Mathematics and Computational Sciences, Shenzhen University, Shenzhen, P. R. China. Email: zbchen@szu.edu.cn}, and
Liqun Qi\thanks{Department of Applied Mathematics, The Hong Kong Polytechnic University, Hung Hum, Hong Kong. Email: liqun.qi@polyu.edu.hk}}

  \makeatletter
      \def\@setcopyright{}
      \def\serieslogo@{}
      \makeatother

\begin{document}

\maketitle

\begin{abstract}
A doubly nonnegative matrix can be written as a Gramian matrix, and a completely positive matrix can therefore be written as a Gramian matrix of some nonnegative vectors.
 In this paper, we introduce Gramian tensors and study 2-dimension completely positive tensors and $\bcp$\  tensors.  Also investigated are the complete positive
 multi-hypergraph, a generalized form of a cp graph. We also provide a necessary and sufficient condition for a 2-dimensional tensor to be completely positive.
\end{abstract}

\noindent \textbf{keywords:} \  Completely positive tensor; $\{0,1\}$ completely positive tensor; multi-hypergraph ;  $(0,1)$ tensor.\\
\noindent \textbf {AMS Subject Classification}: \   53A45, 15A69.  \\


\section{Introduction}
\setcounter{equation}{0}

\indent Completely positive (cp) matrices have been investigated since 1960s\cite{D1962,Hall1963, GW1980, BM2003, Xu2004}, and have been applied in many fields such as
computer vision\cite{SH2005, SH2006}, exploratory multiway data clustering\cite{LD2006},  inequalities\cite{D1962}, quadratic forms\cite{Hall1963}, combinatorial designs
\cite{BM2003} and optimizations\cite{BAD2009, DA2013, AKK2013}.  They are also applied to statistical models\cite{GW1980}. \\
\indent  In 1993, Kogan and Berman use graph theory to character the cp matrices\cite{KB1993}. Meanwhile Salce and Zanardo\cite{SZ1993} use cp matrices to investigate 
the positivity of least squares solutions.  The study on complete positivity reached its peak in the later of 1980s when Berman and Ando etc. began their systematical study on 
cp problems\cite{BM2003}.  In 2004 the first author\cite{Xu2004} presented a sufficient and necessary condition for a square matrix to be cp by the theory of convex cone, and then 
in 2005, together with Berman, he introduced $(0,1)-cp$ matrices, uniform cp matrices and minimal $(0,1)-cp$ matrices\cite{BX2005,BX2007}. Recently cp matrices have been found
useful in pattern recognitions \cite{HPS2005, LD2006} and polynomial optimizations\cite{AKK2013, BAD2009, PVZ2015}. \\
\indent The high order completely positive tensors were introduced by Qi\cite{Qi2013} in 2013.  In fact,  the third-order cp tensor appeared in 2005 as a symmetric tensor possessing
a nonnegative symmetric rank-1 decomposition (also called a cp decomposition in the matrix case),  has been studied in 2005 (in the name of multiway array) by Hazan, Polak and 
Shashua \cite{HPS2005, SH2005,SH2006} where an algorithm is presented for establishing a completely positive decomposition and is applied to image analysis and multiway 
clustering. \\
\indent  A cp tensor is a symmetric tensor, which possesses a symmetric rank-one decomposition\cite{K2009, Qi2013} where each rank-one tensor is nonnegative.  The determination 
of a cp tensor is a NP-hard problem, though, there are some situations when a feasible algorithm possibly exists to settle it\cite{QXX2014}.  Two kinds of positive(nonnegative) tensors
 closely related to cp tensors are \emph{doubly nonnegative tensors}\cite{LQ2016} and \emph{copositive tensors} \cite{Qi2013}.  Like the relationship of a cp matrix and its associated
 graph,  a cp tensor is closely related to a hypergraph\cite{Pe2010, PZ2013} or a multi-hypergraph whose edge-set allows multi-subsets of its vertex-set \cite{PZ2014, XLQC2016}. \\

 \indent  In this paper, we first investigate the cp tensors by introducing the Gramian tensors, and then focus on the even order two-dimensional cp tensors. Then we investigate the 
 $\bcp$\ tensors and the associated cp multi-hypergraphs. \\
 
 \indent  A \emph{doubly nonnegative} (dnn) matrix is both entrywise nonnegative and positive semidefinite (psd).   We denote the set of all dnn matrices of order $n$ by $DNN_{n}$.  
 A matrix $A\in DNN_{n}$ is called \emph{completely positive} (\emph{cp}) if  there exists a nonnegative matrix $W\in \RR^{n\times d}$ for some positive integer $d$ such that
\beq\label{eq: 2ordercp}
A=WW^{\top}
\eeq
The smallest possible number $d$, denoted $cprank(A)$,  is called the \emph{cprank} of  $A$.  $A$ is called $\bcp$ if  $W$ is a $(0,1)$-matrix. The binary cprank of $A$ is accordingly
defined when $W$ is a (0,1) matrix.  (\ref{eq: 2ordercp})  is called a \emph{cp decomposition} of $A$.  A cp matrix may possess many cp decompositions.  Obviously all cp matrices are 
dnn (i.e., $\cp_{n}\subseteq DNN_{n}$  for all $n$) by definition, and $DNN_{n}=\cp_{n}$ for all $n\leq 4$ (this is not obvious).  The inclusion $\cp_{n}\subset DNN_{n}$ becomes 
proper when $n\geq 5$. For more detail on $\cp$ matrices, we refer to \cite{BM2003}. \\

\indent  An $n\times n$ nonnegative matrix $A$ is associated with a graph $G(A) = ([1\ldots n], E)$ such that a pair $(i,j)\in E$ if $a_{ij}$ or $a_{ji}$ is nonzero.  A square real matrix $A$ is
said to be a \emph{realization} of  a graph $G$ if  $G$ is isometric to $G(A)$.  $A$ is called a \emph{dnn} (resp. \emph{cp} and \emph{psd} etc.) \emph{realization} of $G$ if  $A$ is
a dnn (resp. cp  and psd, etc.) matrix,  and also $G(A)=G$.  A graph $G$ is called a \emph{cp graph} if  each of its dnn realizations is a cp matrix. Berman and Hershkowtiz show that a
graph $G$ is cp if and only if  $G$ contains no odd cycle of length greater than 3.  A direct corollary of this result is that a dnn matrix whose associated graph contains no long odd cycle is
cp.  This result is employed to classify $5\times 5$ cp matrices in \cite{Xu2005}.  The problem of determining the complete positivity of a given dnn matrix of order large than four
still remains open \cite{BAD2009,DA2013,SZ1993, BM2003}. The extension of cp matrices to cp tensors is dated back to 2006 when Shuasha and Hazen\cite{SH2005} present an algorithm for nonnegative tensor factorizations and use it to the image analysis. A formal definition for high order completely positive tensor is introduced by Qi in \cite{Qi2013}. \\

\indent We denote $[m\ldots n]:=\set{m,m+1,\ldots, n}$ for any integers $m,n$ satisfying $0\le m\le n$, and $\abs{S}$ for the cardinality of  set (or multiset) $S$, 
$\Zplus^{n}$ for the set of nonnegative integral vectors of dimension $n$, and $\Fb^{n}$ (resp. $\Fb^{n\times n}$) the set of  all (0,1) vectors of dimension $n$  
(matrices of $n\times n$) with $\Fb:=\set{0,1}$.  We also use $\RR^{n}$ to denote the set of real $n$-dimensional vectors and $\RPlus^n$ the set of all nonnegative 
vectors in $\RR^n$.  By $\supp(\bx)$ we mean the support of a vector $\bx$, i.e., the index set of nonzero coordinates of $\bx$.  We use $\T_{m;n}$ to denote the set 
of all \mnrts.  A tensor $\A=(A_{i_{1}\ldots i_{m}})\in \T_{m;n}$ is called \emph{symmetric} if  the values of its entries does not alter under any permutation of its subscripts.
We denote $\SS_{m;n}$ for the set of all \mnsts, $\Fb_{m;n}$ for the set of all $m$th order $n$ dimensional (0,1) tensors, and  $\SF_{m;n}$ the set of  all symmetric tensors in
$\Fb_{m;n}$.  Inherited from \cite{QXX2014}, we write
\[ S(m,n):=\set{\tau=(i_1,i_{2},\ldots,i_m):   i_1,i_{2},\ldots ,i_m\in [1\ldots n]}\]
for the index set of  an element of an $m$-order tensor, and denote
\[ S(k; m,n):=\set{(i_1,\ldots,i_m)\in S(m,n):   i_1+i_{2}+\ldots+i_m=m+k }\]
for $k\in [0\ldots N]$ where $N:=m(n-1)$. An element $\si\in S(m,n)$ is sometimes identified with an $m$-tuple or $m$-multiset or an $m$-permutation chosen 
from $[1\ldots n]$ with displacement allowed. \\

\indent  Let $A=(A_{\si})$ be an \mnst where $\si\in S(m,n)$ and $A_{\si}$ denotes an entry of $\A$ indexed by $\si$.  $\A$ corresponds to an $m$-degree homogeneous
polynomial
\beq\label{eq: mpoly}
f_{\A}(\bx)\equiv \sum\limits_{j=1}A_{i_{1}\ldots i_{m}}x_{i_{1}}\ldots x_{i_{m}}
\eeq
$\A$ is called a \emph{completely positive} or simply a \emph{cp} tensor if  $f_{\A}(\bx)$ can be decomposed as
\beq\label{eq: cp03}
f_{\A}(\bx) = \dsum\limits_{j=1}^{K}(\be_{j}^{\top}\bx)^{m}
\eeq
with $\be_{j}\in \RPlus^{n}$.  If we write $B=[\be_{1},\ldots, \be_{K}]\in \RPlus^{n\times K}$, then  (\ref{eq: cp03}) is equivalent to a decomposition of form
\beq\label{eq: cp}
A=\sum\limits_{j=1}^{K} \be_{j}^{m}, \quad  \be_{j}\in \RPlus^{n}
\eeq
for $\A$, where the smallest possible number $K$ is called the \cprk of $\A$, and is denoted $\cprk(\A)$.  A tensor $\A\in \T_{m;n}$ is called a $\bcp$ tensor
if  $\A$ has a decomposition (\ref{eq: cp}) with $\be_{j}\in \Fb^{n}$, and the corresponding smallest number $K$ is called the \emph{binary cprank} of $\A$,
which is denoted $\cprk_{b}(\A)$.  The following lemma will be used to prove one of the main results (the proof is omitted since it is obvious).\\
\begin{lem}\label{le1:bcprk4diag}
\begin{description}
\item[(1) ]  Let $\D=diag(d_1, \ldots,d_n) \SF_{m;n}$ be a nonnegative integral diagonal tensor.  Then $\D$ is a $\bcp$\ tensor with $\cprk_b(\A)=d_1+d_2+\ldots +d_n$.
\item[(2) ]  Let $\J =\bfe^m$ be the all-ones tensor where $\bfe$ is the all-ones vector of dimension $n$.  Then $\J$ is $\bcp$  with $\cprk_b(\A)=1$.
\end{description}
\end{lem}

\indent  Let $r\in [1\ldots n]$.  A $\bcp$\  tensor $\A$ is called  \emph{$r$-uniform} provided that $\A$ has a decomposition (\ref{eq: cp}) with
$\abs{\supp(\al_{j})}=r$ for all $j\in [1\ldots K]$.  A cp ($\bcp$) tensor is called \emph{minimal cp} (minimal $\bcp$\ ) if it becomes non-cp (non-$\bcp$\ ) 
when any of its diagonal elements is decreased. The minimal cp tensor and uniform cp tensor are both inherited from the matrix case \cite{BX2005}. \\

\indent  Given $\Gam \equiv \set{\al_{1}, \al_{2}, \ldots, \al_{n}}\subset \RR^{r}$ where
\[
\al_{j}=(a_{1j},a_{2j},\ldots,a_{nj})^{\top}.
\]
For any $\si=(i_{1},i_{2},\ldots, i_{m})\in S(m,n)$ where $1\le m\le n$, we denote $\Gam_{\si} \equiv \set{\al_{i_{1}}, \al_{i_{2}}, \ldots, \al_{i_{m}}}$.
The \emph{Hadamard product} of  $\Gam_{\si}$ is vector $\al_{\si}: = \al_{i_{1}}\odot \ldots \odot \al_{i_{m}}\in \RR^{r}$ whose $k$th
coordinate equals $a_{i_{1}k}a_{i_{2}k}\ldots a_{i_{m}k}$ for $k\in [r]$.  The \emph{$m$-inner product} of  $\Gam_{\si}$, denoted
$\La_{\si}=(\al_{i_{1}}, \ldots, \al_{i_{m}})$, is the sum of all coordinates of $\al_{si}$,  i.e.,
\beq\label{eq:definprod}
\La_{\si} = \sum\limits_{i=1}^{n}(\prod_{j=1}^{m} a_{ij})
\eeq
When $\al=\al_{i_{1}}=\ldots =\al_{i_{m}}$,  (\cite{eq:definprod}) is called the $m$-inner product of $\al$.  An $m$-norm of a vector $\al$ is accordingly
defined as
\[ \norm{\al}_{m} := (\al, \ldots, \al)^{1/m}\]
where $(\al, \ldots, \al)$ is the $m$-inner product of $\al$ .  A tensor $\A\in \T_{m;n}$ is called an $m$-order \emph{Gramian tensor} generated by vectors
$\set{\al_{j}}_{j=1}^{n}\subset \RR^{d}$ if it satisfies
\beq\label{eq:gramtensor}
A_{i_{1}i_{2}\ldots i_{m}} = (\al_{i_{1}}, \ldots, \al_{i_{m}}), \quad  \forall \tau:=(i_{1},i_{2},\ldots,i_{m})\in S(m, n)
\eeq
Denote $B:=[\al_1,\al_2,\ldots, \al_n]$. The matrix $B\in \RR^{d\times n}$ is called the \emph{adjacency matrix} of $\A$. \\
\indent  Given any real matrix $B\in \RR^{d\times n}$ and any positive integer $m>1$, we can generate an $m$-order Gramian tensor by the column vectors
of $B$.  For our convenience, we call such a tensor an $m-$Gramian tensor of $B$, and denote it by $\A=\Gram^{(m)}(B)$.  A second order Gramian tensor  $\A$
of matrix $B$ is just a Gramian matrix, i.e.,  $\A=B^{\top}B$. Moreover,  a cp matrix is a Gramian matrix of a nonnegative matrix. \par

\begin{exm}\label{exm: 01}
Let $\D=(D_{i_1 i_2 \ldots i_m})$ be a diagonal tensor of $m$-order $n$-dimension, i.e.,
\[ D_{\si} =\la_{\bar{\si}} \delta_{\si},  \quad  \forall  \si=(i_1, i_2, \ldots, i_m)\in S(m,n)  \]
where $\bar{\si}=(i_1+i_2+\ldots +i_m)/m, \la_{j}\ge 0$ for each $j\in [1\ldots n]$ and $\delta_{i_1 i_2 \ldots i_m}$ is the Kroneck number.  Denote
$D= diag(d_{1}, d_2, \ldots, d_{n})$ with $d_{j}=\la_{j}^{1/m}$ for $j\in [1\ldots n]$.  Then $\D = \Gram^{(m)} (D)$.  Then $\D$ is a completely positive
tensor since $D$ is a nonnegative matrix.  Note that $\cprk(\D)$ is exactly the number of nonzero $\la_j$s.
\end{exm}

\vskip 2mm

\section{Completely Positive Tensors  and $\bcp$  tensors}
\setcounter{equation}{0}

\indent  Let $n>1$ be an integer and $r\in [1\ldots n]$.  An $n\times n$ positive semidefinite (PSD) matrix $A$ of rank $r$ can always be written as a Gramian matrix, i.e.,
$A=\Gram(\al_{1},\ldots, \al_{n})$ for some linearly independent vectors $\al_{1},\ldots, \al_{n}\in \RR^r$.  We sometimes denote $A=\Gram(B)$ where
$B=[\al_1, \ldots, \al_n]\in \RR^{r\times n}$ with $\rank(B)=r$.  Thus a square matrix is cp if and only if it is a Gramian matrix of some nonnegative vectors.
It is shown that the complete positivity in the matrix case is equivalent to double nonnegativity for $n\in [1\ldots 4]$.  This is also conjectured to be true for the case of
high even order tensors.  For this purpose, we consider an even order 2-dimensional doubly nonnegative tensor in this section. As to our knowledge, this kind of tensors 
are very useful in quantum computation. \\
\indent In order to study Gramian tensors and the doubly nonnegative tensors, we recall the H\"{o}lder inequality, which can be restated by the language of $m$-inner
product as

\begin{lem}\label{le: holder}
Let $\al_{1}, \ldots, \al_{m}\in \RPlus^{n}$.  Then
\beq\label{eq:holder}
(\al_{1}, \ldots, \al_{m})^{m} \le \prod_{j=1}^{m} (\overbrace{\al_{j}, \ldots, \al_{j}}^{m})
\eeq
Furthermore, the equality in (\ref{le: holder}) holds when we have
\beq\label{eq:eq4mhold}
\rank(\set{\al_{1},\al_{2},\ldots, \al_{m}})=1
\eeq
\end{lem}

\indent Note that the equality in (\ref{eq:holder}) does not necessarily imply (\ref{eq:eq4mhold}) for $m>2$ unless $rank(\A)=1$, which has been treated in \cite{XLQC2016}.
The following theorem presents a necessary and sufficient condition for an even order tensor to be doubly nonnegative.\\

\begin{thm}\label{th:dnntensor}
 Let  $\A\in \T_{m;n}$ with $m$ an even number. Then  $\A$ is dnn if and only if $\A$ is a Gramian tensor.  Furthermore, if $\A=\Gram^{(m)}(\al_1,\ldots, \al_n)$
 with $\al_j\in \RR^K$, then $\rank(\A)=K$ where $K$ is the smallest possible number.
\end{thm}

\begin{proof}
For sufficiency, we let $\A=\Gram^{(m)}(\al_{1},\al_{2},\ldots, \al_{n})$ where $\al_{j}\in \RR^{K}$ for some positive integer $K$.  Denote
\[ B=(b_{ij})=[\al_{1},\al_{2},\ldots,\al_{n}]^{\top}=[\be_{1}, \be_{2}, \ldots, \be_{K}] \]
Then $B\in \RR^{n\times K}$ where $\be_{j}\in \RR^{n}$ is the $j$th column of $B$ for $j\in [1\ldots K]$.  For any vector $\bx\in \RR^{n}$,
by  the symmetric rank-one decomposition, we have
\beyy
\A \bx^{m} &=& \sum\limits_{i_{1},\ldots,i_{m}} A_{i_{1}\ldots i_{m}}x_{i_{1}}\ldots x_{i_{m}}\\
                   &=& \sum\limits_{i_{1},\ldots,i_{m}} (\sum\limits_{j=1}^{N} b_{i_{1}j}b_{i_{2}j}\ldots b_{i_{m}j}) x_{i_{1}}\ldots x_{i_{m}} \\
                   &=& \sum\limits_{i_{1},\ldots,i_{m}} (\sum\limits_{j=1}^{N} b_{i_{1}j}b_{i_{2}j}\ldots b_{i_{m}j} x_{i_{1}}\ldots x_{i_{m}}) \\
                   &=& \sum\limits_{j=1}^{N} (\sum\limits_{i_{1},\ldots,i_{m}} b_{i_{1}j}b_{i_{2}j}\ldots b_{i_{m}j} x_{i_{1}}\ldots x_{i_{m}}) \\
                   &=& \sum\limits_{j=1}^{N} (\sum\limits_{i=1}^{n} b_{ij}x_{i})^{m}\\
                   &=& \sum\limits_{j=1}^{N} (\be_{j}^{\top}\bx)^{m}.
\eeyy
By definition, $\A$ is completely positive. \par
\indent To prove the necessity, we let $\A\in \CP_{m;n}$.  Then $\A$ can be decomposed as (\ref{eq: cp}) for some nonnegative vectors
$\be_{j}\in \RPlus^{n}, j\in [1\ldots K]$.  Denote $\be_{j}=(b_{1j},b_{2j},\ldots,b_{nj})^{\top}$ for each $j$, and let
$\al_{i}=(b_{i1},b_{i2},\ldots, b_{iK})^{\top}$ for each $i=1,2,\ldots,n$. Then  $\al_{i}\in \RPlus^{K}$.
Now given any $\tau:=(i_{1}, i_{2},\ldots,i_{m})\in S(m,n)$, by (\ref{eq: cp}) , we have
\beyy
A_{i_{1}i_{2}\ldots i_{m}} & = & (\sum\limits_{j=1}^{K} \be_{j}^{m})_{i_{1}i_{2}\ldots i_{m}} \\
                                                & = & \sum\limits_{j=1}^{K} b_{i_{1}j}b_{i_{2}j}\ldots b_{i_{m}j}\\
                                                & = & (\al_{i_{1}},\al_{i_{2}},\ldots, \al_{i_{m}})
\eeyy
It follows that $\A$ is the $m$th order Gramian tensor of vectors $\al_{1},\al_{2},\ldots, \al_{n}$.  The proof is completed.
\end{proof}

\indent As a corollary of Theorem \ref{th:dnntensor}, we have
\begin{cor}\label{cor:cptensorgram}
 Let  $\A\in \ST_{m;n}$.  Then $\A$ is a cp tensor if and only if  it is an $m$th order Gramian tensor of some nonnegative vectors, i.e.,
 there exist some nonnegative vectors  $\al_1, \al_2,\ldots, \al_n\in \RPlus^d$ such that  $\A=\Gram^{(m)}(\al_1,\al_2,\ldots, \al_n)$, and the smallest possible
 number $d$ is the cprank of $\A$.
\end{cor}

\indent  For any $r\in [0\ldots m]$, we denote by $S^{r}(m,2)$ (or simply $S^r$ if no risk of confusion arises) the subset of $S(m,2)$ consisting of the 
elements $\si:=(i_1, i_2,\ldots, i_m)$ which have exactly $r$ ones (and thus $m-r$ 2s).  Thus $S^0$ and $S^m$ are resp. the singleton of all-1 $m$-tuple and  
all-2 $m$-tuple. It is obvious that the set of all $S^{r}$ ($r\in [0\ldots m]$) forms a partition of $S(m,2)$. \\
\begin{prop}\label{prop14sst}
Let $\A\in \T_{m,2}$. Then $\A$ is symmetric if and only if  $A_{\si}$ is constant for each $\si\in S^r$  given any $r\in [m]\cup\set{0}$.
\end{prop}
\indent  We use $a_r$ to denote the constant $A_{\si}$ when $\si\in S^r$ for $i\in [0\ldots m]$.  Then there are at most $m+1$ distinct elements in $\A\in \ST_{m,2}$. \\
\indent  Now we investigate $\bcp$\ tensors.  We already know that a cp ($\bcp$\ ) tensor is always strong symmetric, and a $\bcp$\  tensor is surely a nonnegative integral
tensor, i.e., each of its entries is a nonnegative integer.  A natural question arises: \\
\begin{prob}label{prob1}
When does a nonnegative integral strong symmetric tensor become $\bcp$ ?
\end{prob}
\indent In the following we will describe such a tensor with dimension two, which can be regarded a special case for the hierarchical diagonally dominant
tensor\cite{QXX2014}. \par

\begin{thm}\label{th: bcpdim2}
 Let  $\A=(A_{i_{1}i_{2}\ldots i_{m}})\in \SS_{m;2}$ be a nonnegative integral tensor, each $i_k$ taking values either 1 or 2.  Then $\A$ is $\bcp$  if and only if  each
 off-diagonal element is dominated by the corresponding diagonal element, i.e.,
\beq\label{eq: dominate}
A_{i_{1}i_{2}\ldots i_{m}} \le  A_{i_{k}i_{k}\ldots i_{k}}, \quad  \forall  k\in [m]
\eeq
Furthermore, we have
\beq\label{bcprank}
\cprk_{b}(\A)=A_{11\ldots 1}+A_{22\ldots 2}-A_{11\ldots 12}
\eeq
\end{thm}

\begin{proof}
For sufficiency, we suppose that $\A$ is $\bcp$.  Then by (\ref{th:dnntensor}) $\A$ is a Gramian tensor, i.e., $\A=\Gram(\al_1,\al_2)$, where  $\al_1,\al_2\in \SF^N$ with
$N>1$ a positive integer.  Denote  $S_j=\supp(\al_j)$ ($j=1,2$).  For any given $\tau:=(i_{1}, i_{2}, \ldots, i_{m})\in S(m,2)$.  The inequality (\ref{eq: dominate})
is trivial if $i_{1}= i_{2}= \ldots= i_{m}$ (equals 1 or 2).  Now consider the case when $i_{1}, i_{2}, \ldots, i_{m}$ are not identical. In this case, we have
$B(\tau)=\set{1,2}$. Thus we have $A_{\tau}=\abs{S_1\cap S_2}\le \abs{S_j}=A_{jj\ldots j}$ for $j=1,2$.  This proves inequality (\ref{eq: dominate}).\\
\indent  For the sufficiency, we suppose $\A$ is a strong symmetric nonnegative integral 2-dimensional tensor satisfying inequality (\ref{eq: dominate}).  We need to
show that $\A$ is a $\bcp$\  tensor.  For convenience, we denote
\[ n_1=A_{11\ldots 1},\quad  n_2=A_{22\ldots 2},\quad  n_{12}=A_{122\ldots 2}\]
and let $p=n_1+n_2-n_{12}$.  Then by (\ref{eq: dominate}) we have $n_i\ge n_{12}$ and hence $p\ge n_i$ for $i=1,2$.  Now we generate two (0,1)-vectors
$\al_1, \al_2\in \SF^p$ such that $S_i =\supp(\al_i)$ (i=1,2) with
\[ S_1 =[n_1],  \quad  S_2 = \set{1, 2,..., n_{12}, n_1+1,n_1+2,\ldots, n_1+(n_2-n_{12})}  \]
We can check easily that $\A =\Gram(\al_1, \al_2)$. \\
\indent Now we show that the binary cprank of $\A$ is exactly $p$.  From the construction of $\al_1, \al_2$, we have
$\cprk_{b}(\A)\le p=n_1+n_2-n_{12}$ since
\[ \cprk_{b}(\A) \le \abs{\supp(\al_1)\cup \supp(\al_2)} = n_1+n_2-n_{12} =p \]
Now write $\D = \A -n_{12}\bfe^m, \bfe=(1,1)^{\top}$. Then $\D$ is diagonal. By using Lemma \ref{le1:bcprk4diag} we get
$\cprk_b(\A) = tr(\D) +n_{12} = (A_{11\ldots 1}-n_{12}) + (A_{22\ldots 2}-n_{12}) +n_{12} =n_1+n_2-n_{12}=p$.
\end{proof}

\indent  For any $\si\in S(m,n)$,  a diagonal element $A_{ii\ldots i}$ is \emph{associated with $A_{\si}$}  if  $i\in B(\si)$.  By Theorem \ref{th: bcpdim2} we have

\begin{cor}\label{cor: bcpnecessity}
Let  $\A\in \SS_{m;n}$ be $\bcp$.   Then each of its off-diagonal entries is no larger than any of its associated diagonal entries,  i.e.,
\beq\label{eq: dominate}
A_{i_{1}i_{2}\ldots i_{m}} \le  A_{i_{k}i_{k}\ldots i_{k}}, \quad  \forall  k\in [m]
\eeq
\end{cor}

\begin{proof}
Let $\tau:=(i_{1}, i_{2}, \ldots, i_{m})\in S(m,n)$ and let $k\in S=B(\tau)$.  Then $A_{kk\ldots k}$ is an associated diagonal element with $A_{\tau}$.
For the case $\abs{B(\tau)}=1$, the inequality (\ref{eq: dominate})  is trivial.   Now suppose $\abs{B(\tau)}\ge 2$ and let $\set{j,k}\subseteq S$.
Consider the sub-tensor $\A_1$ induced by the index set $\set{j,k}$.  Then $\A_1$ is also $\bcp$\  (\cite{XLQC2016}).  The result follows by
Theorem \ref{th: bcpdim2}.
\end{proof}

\indent  We note that condition (\ref{eq: dominate})  is also a sufficient condition to guarantee general complete positivity as stated in the following:

\begin{thm}\label{th: cpdim2}
Let  $\A\in \SS_{m; 2}$ be nonnegative.   Then $\A$ is completely positive if  for each $\si\in S(m;n)$
\beq\label{eq: dom4gcp}
A_{\si} \le  \min \set{A_{ii\ldots i} | i \in B(\si)}
\eeq
Furthermore, $\cprk(\A)\le 3$,  and $\cprk(\A)=3$ if and only if each diagonal element $A_{ii\ldots i}$ is larger than any of  off-diagonal elements.
\end{thm}

\begin{proof}
There are at most three distinct values for a strong symmetric $m$-order 2-dimensional  tensor $\A$, i.e.,
\[ a_1:=A_{11\ldots 1}, \quad  a_2:=A_{22\ldots 2}, \quad  a_3:=A_{11\ldots 12}, \]
and all off-diagonal elements are equal to $a_3$.   Thus the condition (\ref{eq: dom4gcp}) is equivalent to
\beq\label{eq2:dom4cp}
 0\le a_3\le \min(a_1, a_2)
\eeq
If $a_3=0$, then the result is obvious since a nonnegative diagonal tensor is completely positive as mentioned in \cite{XWL2016}.
For $a_3>0$, we may set
\[ b_1 =(a_1 -a_3)^{1/m},  \quad    b_2 =(a_2 -a_3)^{1/m}, quad  b_3 = a_3^{1/m}. \]
Then $b_j$'s are all nonnegative real numbers due to condition (\ref{eq2:dom4cp}).  We set
\[ \be_1 =[b_1, 0]^{\top},  \be_2 =[0, b_2]^{\top}, \be_3 =[b_3, b_3]^{\top}.  \]
Then we can verify that $\A = \be_1^m +\be_2^m +\be_3^m$.  So $\A$ is completely positive, with $\cprk(\A)\le 3$. \\
\end{proof}

\indent  We shall mention that Theorem \ref{th: cpdim2} can also be proved by using our result in \cite{QXX2014}.  Unfortunately condition (\ref{eq: dom4gcp})
is not necessary for a tensor to be cp.  This fact can be illustrated by consider the following example: \\
\begin{exm}\label{exm4gcp}
Consider $m=2$ and let
\[  A = \begin{pmatrix}
      1  & 2    \\
      2 &  5
\end{pmatrix} \]
It is easy to check that $A$ is a completely positive tensor (of order-2 dimension-2) since $A=BB^{\top}$ if we take
\[  B = \begin{pmatrix}
      1  & 0    \\
      2 &  1
\end{pmatrix} \]
\end{exm}

\indent  Recall that a \emph{slice} of a tensor $\A\in \T_{m;n}$ is a sub-tensor of order $m-1$ obtained by fixing one of the subscripts. For example, a slice of a
3-order tensor along one of the three directions, say the first, i.e., $A(i,:,:)$, is a matrix.  A \emph{zero slice}(all of whose entries are zero) is called a 
\emph{trivial} slice. Given a nonempty subset $\cI:=\set{s_{1}, s_{2},\ldots, s_{r}}$ of $[1\ldots n]$, a \emph{principal subtensor} $\A[\cI]$ of $\A$ induced 
by $\cI$ is an $m$-order $r$-dimensional tensor $\B=(A_{i_{1}i_{2}\ldots i_{m}})$ whose indices $i_{k}$s are all constrained in $\cI$.   A \emph{zero block} 
is a principal subtensor whose entries are all zero. An irreducible tensor has no zero slice nor any zero block. \\
\indent  It is pointed out in \cite{XLQC2016} that all the slices and the induced principal subtensors of a $cp$ (binary $cp$) tensor are also $cp$ (binary $cp$).  
Based on this point, we present a necessary condition, which is weaker than (\ref{eq: dom4gcp}), for a tensor to be cp. \\

\begin{thm}\label{th: gcpnec}
 Let  $\A\in \SS_{m; n}$ be a cp tensor.   For any $\tau\in S(m, n)$ with $B(\tau)=\set{i, j}$, we have
\beq\label{eq: dom4gcp}
A_{\tau}^2 \le  A_{ii\ldots i}A_{jj\ldots j}
\eeq
\end{thm}

\begin{proof}
Let $\tau:=(i_1, i_2, \ldots, i_m)\in S(m,n)$ with $B(\tau)=\set{i,j}\subseteq [1\ldots n]$.  If $i=j$, then inequality (\ref{eq: dom4gcp}) is obvious.  Thus in the following we
may assume that $1\le i< j\le n$, and take $\cI=\set{i,j}$. Then the induced subtensor $\A[\cI]$ is a 2-dimensional completely positive tensor.  We are now confined to
$\A_1:=\A[\cI]$.  Since $\A\in \SS_{m;2}$ is completely positive, there exist some nonnegative vectors $\al_1,\al_2\in \RPlus^N$ ($N=\cprk(\A_1)$) such that
$\A_1 = \Gram(\al_1,\al_2)$.  It follows that $A_{\tau} =(\al_{i_1}, \al_{i_2},\ldots, \al_{i_m})$ where $B(\tau)=\cI$.  By Lemma \ref{le: holder} we have
\beq\label{eq4prf01}
A_{i_{1}i_{2}\ldots i_{m}}^m \le \prod_{k=1}^{m} A_{i_{k}i_{k}\ldots i_{k}}
\eeq
where $i_k$ takes value in $\cI=\set{i,j}$.  Denote $\tau_i=(i,i,\ldots,i, j), \tau_j=(i, j,\ldots, j, j)$.  Then
we have $A_{\tau}=A_{\tau_i} =A_{\tau_j}$ since $B(\tau)=B(\tau_i)=B(\tau_j)=\set{i, j}$ and $\A$ is strong symmetric.
By (\ref{eq:hold2A}) we have
\beq\label{eq1}
A_{ii \ldots i j}^m \le A_{ii\ldots ii}^{m-1}A_{jj\ldots jj}
\eeq
and
\beq\label{eq2}
A_{ij \ldots jj}^m \le A_{ii \ldots ii} A_{jj \ldots jj}^{m-1}
\eeq
Since $A_{\tau_1}=A_{ii\ldots ij}=A_{i j\ldots jj}=A_{\tau_2}$, we have by (\ref{eq1})  and (\ref{eq2})
\[ A_{\tau}^{2m} =A_{\tau_1}^m A_{\tau_2}^m \le (A_{ii\ldots ii} A_{jj\ldots jj})^m \]
which is followed by (\ref{eq: dom4gcp}) .
\end{proof}

\indent  It is not clear yet whether (\ref{eq: dom4gcp}) is also sufficient for an 2-dimensional nonnegative strong symmetric tensor
to be cp. \\

\section{Completely positive multi-hypergraphs}
\setcounter{equation}{0}

\indent  Let $\A\in \SS_{m;n}$ be a $k$-uniform $\bcp$ tensor and let $\G=(V,\capEb)$ be a \pdg\  associated with $\A$.  Denote
$A=[\al_{1},\ldots, \al_{r}]\in \Fb^{n\times r}$ (each $\al_i$ corresponds to a unique maximal multi-edge of $\G$).  Then $\A$ can
be expressed as the $m$th-power of $A$, denoted by $A^{\odot m}$, in the sense of \emph{Khatri-Rao product}\cite{K2009}, or briefly
an $m$-KR power of $A$,  i.e.,
\[ A^{\odot m}=\overbrace{A\odot A\odot \ldots \odot A}^{m} \]
where product $\odot$ is defined as the columnwise Kroneck product.  We say that $\A:=A^{\odot m}$ has $r$ \emph{$k$-uniform components}
$\al_{j}^{m}$. $\A$ is sometimes written as $\A=\sum A^{\odot m}$ where
\[ \sum A^{\odot m}:= \dsum_{j=1}^{r} \al_{j}^{m} \]
and $A\in \RPlus^{n\times r}$ is an $k$-uniform $\set{0,1}$ matrix.  The number $k$ is called the \emph{support }of $\A$ and denoted by $\supp(\A)$. \\

\begin{thm}\label{le: kbasegrph}
Let $\A=\sum(A^{\odot m})\in \T_{m;n}$ be $m$-uniform ($2\le m\le n$) and $\bcp$ with $A=[\al_{1},\ldots,\al_{r}]$.  Let $\G=(V,\capEb)$ be the
multi-hypergraph associated with $\A$.  Then
\begin{enumerate}
\item If $\A$ is a (0,1) tensor,  then $n=mr$ and $\cprk(\A) \le \cprk_{b}(\A)\le \frac{n}{m}$.
\item If $\A$ is an essential (0,1) tensor, then $\cprk(\A) \le \cprk_{b}(\A)\le \lceil \frac{n}{k-1}\rceil$.
\end{enumerate}
\end{thm}

\begin{proof}  (1). It is obvious that $\cprk(\A) \le \cprk_{b}(\A)$ since $\cprk(\A)$ is the least number for all the possible symmetric nonnegative
decomposition (\ref{eq: cp}), including the $\bcp$ decomposition.  Denote $S_{j}=\supp(\al_{j})$ for each $j\in [r]$. We want to show the second part of Item (1).
Suppose $\A$ is a (0,1) tensor and there is a pair of positive integers $(p, q) (1\le p < q\le r$) such that $S_{p}\cap S_{q}\neq \emptyset$.  We may assume that
$k\in S_{p}\cap S_{q}$, i.e.,  $a_{kp}=a_{kq}=1$ where $a_{ij}$ is the $i$th coordinate of $\al_{j}$.  Therefore
\[  A_{kk\ldots k}=\sum\limits_{j=1}^{r} a_{kj}^{m} \ge a_{kp}^{m} +a_{kq}^{m} =2 \]
a contradiction to our assumption that $\A$ is $(0,1)$.  Thus $S_{p}\cap S_{q}=\emptyset$ for all distinct $p, q\in [r]$.  It follows that
$\set{S_{1},S_{2},\ldots, S_{r}}$ forms a uniform division of $[1\ldots n]$ with each part possessing $m$ elements.  So $mr=n$, and thus
$\cprk_{b}(\A)\le r=\frac{n}{m}$. \\
\indent (2).  Let $\A$ be $\bcp$ and also an essential (0,1) tensor.  We now show that $\abs{S_{i}\cap S_{j}}\le 1$ for all distinct $i,j\in [r]$.
In fact, if there is a pair $(p, q): 1\le p < q\le r$ such that $\abs{S_{p}\cap S_{q}}\ge 2$, then there exist two different numbers $s,t\in [1\ldots n]$ such that
$s,t\in S_{p}\cap S_{q}$.  We show that there exists some $\si\in S(m,n)$ such that $A_{\si}\ge 2$.  Actually if we choose
$\si=(s,t,t,\ldots, t)\in S(m,n)$, then by (\ref{eq: cp}) we have
\beyy
A_{\si} \ge (\al_{i}^{m})_{\si} + (\al_{j}^{m})_{\si} =a_{si}a_{ti}^{m-1}+ a_{sj}a_{tj}^{m-1} =2
\eeyy
The last equality is due to the fact that $s,t\in S_{i}$ implies $a_{si}=a_{ti}=1$ and $s,t\in S_{j})$ implies $a_{sj}=a_{tj}=1$.  This complete the proof.
\end{proof}

\indent  Now we assume $\G=\G(\A)$ be a \pdg (without isolated vertices) associated with an essential (0,1) tensor $\A\in \SS_{m;n}$ with a decomposition
(\ref{eq: cp}) where each $\al_{j}$ is a (0,1) $n$-dimensional vector.  Denote $\A^{*}$ as the pattern of $\A$, i.e.,  $a^{*}_{\si}=1$ if  $a_{\si}\neq 0$
for any $\si\in S(m,n)$.  Then $\A$ is permutation similar to a direct sum of some irreducible tensors \cite{XLQC2016}, say,
\[ \A \sim_{p} \A_{1}\oplus \A_{2}\oplus \ldots \oplus \A_{r}  \]
where $\A_{j}\in \SF_{m,n_{i}}$ with $n_{1}+\ldots +n_{r+1}=n$. Here each $\A_{i}$ corresponds to a complete block. But the essential (0,1) tensor, if it is $\bcp$, associated
with a complete block is a tensor each of whose elements except the diagonal ones is 1. Now we consider any nonnegative tensor $\A\in \T_{m;n}$. If $\A$ is $\bcp$, then
$\A$ has a decomposition (\ref{eq: cp}) where $\al_{j}\in \Fb^{n}$ for each $j\in [r]$. An edge $\si=\set{i_{1},\ldots,i_{m}}\in S(m,n)$ is called a \emph{maximal edge} of a multi-hypergraph $\G=(V, \capEb)$ if $\G$ has no edge $\eps$ such that $B(\si)\subset B(eps)$.
We call a multi-hypergraph $\G$ an $r$-uniform multi-hypergraph if all its maximal edges have cardinality $r$.  A multi-hypergraph $\G=(V, \capEb)$ is said to have Property
$R$ if  $\D_{\al} \subseteq \capEb$ for any $\al\in \capEb$ where
\beq\label{eq: D4R}
 \D_{\al}=\set{\si\in \capEb: B(\si)\subseteq B(\al)}
 \eeq
Property $R$, first introduced in \cite{XLQC2016}, implies that $\G$ is uniquely determined by the set of its maximal edges. \\

\begin{exm}\label{exm04}
Let $\G$ be a 3-uniform 3-order multi-hypergraph with $V=\set{1,2,3,4}$ and a unique maximal edge $E_{m}=\set{1,3,4}$. Then the (multi-)edges of $\G$ are the following
\[ \set{1,3,4}, \set{1,1,3}, \set{1,3,3},\set{1,1,4},\set{1,4,4} \]
\[ \set{3,3,4}, \set{3,4,4}, \set{1,1,1},\set{3,3,3},\set{4,4,4}.\]

\noindent $\G$'s adjacent tensor is a strong symmetric (0,1)$ tensor \A=(A_{ijk})\in \Fb^{4\times 4\times 4}$ defined as
\beyy
    A(:,:,1) &=\left[\begin{array}{cccc} 1&0&1&1\\  0&0&0&0\\  1&0&1&1\\  1&0&1&1\end{array}\right], \quad &  A(:,:,2)=0 \\
    A(:,:,3) &=\left[\begin{array}{cccc} 1&0&1&1\\  0&0&0&0\\  1&0&1&1\\  1&0&1&1\end{array}\right], \quad &
   A(:,:,4) =\left[\begin{array}{cccc} 1&0&1&1\\  0&0&0&0\\  1&0&1&1\\  1&0&1&1\end{array}\right]
\eeyy
It is easy to verify that $\A=\al_{1}^{3}$ where $\al=(1,0,1,1)^{\top}$.
\end{exm}

\indent  A \pdg\ $\G$ is called a \emph{cp pseudograph} if its adjacency tensor $\A(\G)$ is $\bcp$.  In Example \ref{exm04} the \pdg\ $\G$ has 10 edges, among
which there is one \emph{normal edge} $E=\set{1,3,4}$, which is also a maximal edge.  Actually for any $m$-uniform \pdg\  $\G$ of size $n$, the largest number of
maximum normal edges is ${n\choose m}$ among its $n^{m}$ (multi-)edges.  Consider the ratio of the number of normal $m$-edges to the number of (multi-)edges, i.e.,
\[
R(m,n) := \frac{{n\choose m}}{n^{m}}.
\]
$R(m,n) $ converges to $R_{m}:=\frac{1}{m!}$ when $n\to \infty$.  For example,  $R_{3}=1/6, R_{4}=1/24, R_{5}=1/120, \ldots$.
This implies that a \pdg\ is much more complicated than a hypergraph. \\

\begin{cor}\label{cor05}
Let $\G=(V, \capEb)$ be an $m$-order multi-hypergraph with $V=[1\ldots n]$ with $m, n >1$.   If $\G$ possesses property $R$ and has a unique nonempty maximal edge,
then $\G$ is a cp \pdg.
\end{cor}

\begin{proof}
Denote $\G=(V,\capEb)$ and assume that the unique maximal edge $E$ has the base set
\[
B(E)=\set{i_{1},i_{2},\ldots,i_{r}},\quad  1\le i_{1}<i_{2}<\ldots <i_{r}\le n, 1\le r\le m. \]
Let $\A=\A(\G)=(A_{\si})\in \Fb_{m;n}$ be its adjacency tensor.  Then $\A$ is a (0,1) tensor.  We now prove that $\A$ is actually a $\bcp$ tensor
with $\cprk_{b}(\A)=1$, that is, there is a (0,1) vector $\al\in \Fb^{n}$ such that  $\A=\al^{m}$.  For this purpose, we denote
\[ D_{E}=\set{\si=(j_{1},j_{2},\ldots,j_{m}): B(\si)\subseteq B(E)}  \]
and let $\al=(a_{1},a_{2},\ldots,a_{n})^{\top}\in \Fb^{n}$ such that  $\supp(\al)=B(E)=\set{i_{1},i_{2},\ldots,i_{r}}$.  Thus $a_{j}=1$ if and only if $j=i_{k}$ for
some  $k\in [r]$.  It suffices to show that $\A=\al^{m}$, i.e.,
\beq\label{cor05prf}
A_{j_{1}j_{2}\ldots j_{m}}=a_{j_{1}}a_{j_{2}}\ldots a_{j_{m}} \quad \forall  \si=(j_{1},j_{2},\ldots, j_{m})\in S(m,n)
\eeq
In fact, since  $a_{j_{1}}a_{j_{2}}\ldots a_{j_{m}}=(\al^{m})_{\si}=1$ for any $\si=(j_{1},j_{2}, \ldots, j_{m})\in S(m,n)$. 
It follows that $B(\si)\subseteq \supp(\al)=B(E)$ and thus $\si\in D_{E}$, which implies $\si\in \capEb$ (since $\A$ has Property $R$).  
Consequently we have $A_{\si}=1$.\\
\indent Conversely we let  $A_{\si}=1$ for some $\si=(j_{1},j_{2},\ldots,j_{m})\in S(m,n)$, which is equivalent to $\si\in \capEb$. Thus 
\[ B(\si)\subseteq B(E)=\set{i_{1},i_{2},\ldots,i_{r}}=\supp(\al) \]
since $E$ is the unique maximal edge of $\G$.  It follows that 
\[ (\al^{m})_{\si}=a_{j_{1}}a_{j_{2}}\ldots a_{j_{m}}=1. \]
The proof is completed.
\end{proof}

\indent  A \pdg\ $\G$ with $n$ vertices and $N$ edges is called an $n\times N$ \emph{multi-hypergraph}.  Usually the maximal edges are not unique.
In \cite{XLQC2016}, we define the \emph{indicator} of  an edge $\al$ of $\G$ as the vector $I_{\al}:= (w_{1},\ldots, w_{n})^{\top}$ in $\Zplus^{n}$ where $w_{i}$
denotes the frequency of vertex $i$ in $\al$.  An $n\times N$ \pdg\ $\G$ is uniquely determined by an $n\times N$ nonnegative integral matrix
\[  W=W(\G):= [\bu_{1},\ldots, \bu_{N}]  \]
where $\bu_{j}\in \Zplus^{n}$ is the indicator of $\al_{j}\in \capEb$.  $W$ is called \emph{the adjacency matrix} of $\G$.  Now we form matrix $A$ associated with
$W$ by
\beq\label{asscpmatrx}
 A = WW^{\top}
\eeq
$A$ can be written equivalently as
\[
A = \sum\limits_{j=1}^{N}  \bu_j^2 = \sum\limits_{j=1}^{N}  \bu_j \bu_j^{\top}  \]
which is exactly a $\bcp$\  matrix when each $\bu_j$ is a (0,1) vector (\cite{BX2005}).  $A$ is called an $k$-uniform cp matrix if $\abs{\supp(W)}=k$, and $\A$
is called an $k$-uniform $n\times m$ tensor of rank $R$ if $\A$ has a $\bcp$\ decomposition (\ref{eq: cp}).  \\

Denote
\[ \C_{\al} :=\set{ \be\in \capEb:  \be \sim \al}, \quad \texttt{and}\  \D_{\al} :=\set{ \be\in \capEb:  \be \prec \al } \]
for any edge $\al\in \capEb$.  Let $\Gamma_{\G}: =\set{\al_j  | j=1,2,\ldots, r}$ be the set of the maximal edges of $\G$.  Then
$\set{\D_{\al_j}: j=1,2,\ldots, r}$ forms a partition of $\capEb$.  A \pdg\ $\G=(V,\capEb)$ is said to have Property $R$ if  $D_{\al}\subseteq \capEb$ whenever $\al\in \capEb$.
In \cite{XLQC2016} we show that a (0,1) \mnst  $\A$ is $\bcp$ if and only if  $\cP$ possesses Property
$R$ where $\cP=\cP(\A)$.  We have shown in \cite{XLQC2016} that a (0,1) tensor $\A$ is $\bcp$  if and only if  $\A$ can be written as the direct sum of some all-ones blocks. This is
equivalent to
\beq\label{eq: disjointsupp}
S_{i} \cap S_{j} = \emptyset, \forall 1\le i < j\le r
\eeq
where $S_{k}:=\supp(\bu_{k})$ and $r$ is the smallest number for (\ref{eq: cp}) to hold. \\
\indent Given an $n\times N$ \pdg\ $\G$.  We let $\A$ denote the tensor generated by the
Khartry-Rao product of  $W\equiv W(\G)=[\bu_1,\ldots, \bu_N]$, i.e.,
$\A = \overbrace{W\circ W\circ \ldots \circ W}^m$, which is defined as (\ref{eq: cp}).\\ 
\indent We note that a $\bcp$ tensor may not be a (0,1) tensor, and a (0,1) tensor can be a non-$\bcp$ even though it is a cp tensor.\\

\indent  A multi-hypergraph $\cP =(V, \capEb)$ is called an $m\times n$ \pdg\ if
$\abs{V}=n$ and each edge of $\cP$ is an $m$-multiset of $V$. Let $\capEb_{k}$ be
the subset of $\capEb$ each of whose largest elements has exactly $k$ distinct elements.
We let $\cP_{k}:=(V, \capEb_{k})$. For an $k$-uniform CP tensor $\A\in \cp_{m;n}$, its
associated \pdg $\cP$ always has a $k$-base.\\
\indent  Given a tensor $\A=(A_{\si})\in \T_{m;n}$,  A \emph{tensor pattern} $\tA=(\tilde{A}_{\si})\in \Fb_{m,n}$ is a (0,1) tensor satisfying
\[
\tilde{A}_{\si}=1\Leftrightarrow A_{\si}\neq 0, \quad   \forall  \si\in S(m,n)
\]
An \mnrt $\A=(A_{\si})\in \T_{m;n}$ is called a \emph{reducible} tensor if there is a proper subset $\cI\subset [1\ldots n]$ such that
\beq\label{eq:defreducible}
a_{i_{1}\ldots i_{m}} =0, ~~\forall i_{1}\in \cI,~\forall i_{2}, \ldots,i_{m} \notin \cI.
\eeq
$\A$ is called \emph{irreducible} if it is not reducible.

\indent  Reducibility is a pattern property for tensors. By employing the permutational similarity property, we can decompose any $(0,1)$ reducible tensor into a direct sum of a finite number of low dimensional irreducible tensors and a zero tensor in the permutational similar sense. Before stating this result, some related concepts are recalled here. Let $\A, \B\in \T_{m;n}$. We say that $\A$ is \emph{permutational similar} to $\B$, denoted as $\A\sim_{p} \B$,  if there exists a permutation matrix $P\in RR^{n\times n}$
such that
\[ \B = \A\times_{1} P\times_{2}  P\times_{3} \cdots \times_{m} P, \]
where $\tA:=\A\times_{k} P=(\tilde{a}_{i_{1}\ldots i_{m}})\in \T_{m;n}$ is defined as
\[ \tilde{a}_{i_{1}\ldots i_{k-1}i_{k}i_{k+1}\ldots i_{m}} = \sum\limits_{j=1}^{n} a_{i_{1}\ldots i_{k-1}j i_{k+1}\ldots i_{m}}p_{i_{k}j} \]

\vskip 2mm

Utilizing the permutational similarity of tensors, we can build up some identical relation among their corresponding multi-hypergraphs. Let $\cP_{1}=(V_{1}, \capEb_{1})$ and $\cP_{2}=(V_{2}, \capEb_{2})$ be two given $m$-uniform multi-hypergraphs with their $(0,1)$ associated tensors $\A$ and $\B$ respectively.
Then $\A\sim_{p} \B$ if and only if there exists a bijection $\phi$ from $V_{1}$ to $V_{2}$ such that
\[  \{i_{1},i_{2},\ldots,i_{m}\}\in \capEb_{1}  \mapsto \{\phi(i_{1}),\phi(i_{2}),\ldots,\phi(i_{m})\}\in \capEb_{2} \]
that is, $\cP(\B)$ is the multi-hypergraph obtained from $\cP(\A)$ by the reordering of its vertices, and thus they are identical in this sense. \\

Let $\A_{i}=(a_{\si}^{(i)})\in \T_{m,n_{i}}, i=1,2$ and $n_{1}+n_{2}=n$.  The \emph{direct sum} of $\A_{1}$ and $\A_{2}$, denoted by
\[ \A = \A_{1}\oplus \A_{2} =(a_{i_{1}\ldots i_{m}}), \]
is defined by
\[ a_{i_{1}\ldots i_{m}}=\begin{cases} a_{i_{1}\ldots i_{m}}^{(1)}  & \text{if} i_{1},\ldots, i_{m}\in [n_{1}], \\
a_{i_{1}\ldots i_{m}}^{(2)}  & \text{if} i_{1},\ldots, i_{m}\in n_{1}+[n_{2}], \\
      0& \text{otherwise}.
\end{cases}  \]
Here $a+S$ is defined as the translation of set $S$,  i.e., $a+S=\set{a+s: s\in S}$.\\

\vskip 2mm

Now we are in a position to describe the decomposition for tensors in the sense of permutation similarity.

\begin{lem}\label{lem1}
Let $\A\in \SF_{m,n}$, where $m\ge 2, n\ge 1$. Then
\beq\label{eq:irredufact}
\A\sim_{p} \A_{1}\oplus \A_{2}\oplus \ldots \oplus \A_{r}\oplus \mathcal{O}_{r+1}
\eeq
where $\A_{i}\in \SF_{m,n_{i}}$ is irreducible, $\mathcal{O}_{r+1}$ is a zero tensor of order $m$ and dimension $n_{r+1}$, and $n_{1}+\ldots + n_{r+1}=n$.
\end{lem}

\begin{proof} The result is trivial if $\A$ is irreducible tensor. Now we assume that $\A\in \SF_{m,n}$ is a reducible tensor. We will use induction to prove the desired statement. For $n=1$, the reducibility implies that $\A=0$. The statement holds by setting $r=0$. Assume that for all $k$ satifying $1\leq l\leq n$ with $n \geq 1$, the statement holds. They for the case of $l+1$, there exists a nonempty subset $I$ of $[l+1]$ such that
\beq\label{eq:leprf1}
 A_{i_{1}i_{2}\ldots i_{m}}=0, \forall i_{1}\in I, i_{2},\ldots, i_{m}\notin I
\eeq
Let $\cP=(V, \capEb)$ be the multi-hypergraph with $\A$ as an associated tensor, and we assume w.l.g. that
\[ I:=\set{k_{1}, k_{2},\ldots, k_{r}}, 1\le k_{1} <k_2 <\cdots < k_r\le l+1.\]
Then we let $\phi: [l+1]\to [l+1]$ be an one-to-one correspondence such that
\[ \phi(k_{i})=i, \quad  \forall\  i=1,2,\ldots, r. \]
and $\phi$ maps $[l+1]\backslash I$ to $[l+1]\backslash [r]$.   $\phi$ can be regarded as a permutation on $[l+1]$, and so there is a permutation matrix $P$ corresponding to $\phi$.
Actually if we define $P=(p_{ij})\in \Fb^{(l+1)\times (l+1)}$ by
\[ p_{ij}=1 \quad \texttt{iff}\  j=\phi(i)  \]
for each $i\in [l+1]$. It follows readily that
\beq\label{eq:directsum}
\tA:= \A\times_{1}P\times_{2}P\times_{3} \ldots \times_{m}P  = \A_{11}\oplus \A_{22}
\eeq
where $\A_{11}\in \SF_{m,r}, \A_{22}\in \SF_{m,l+1-r}$. Note that $r$, $l+1-r\leq n$,  the desired decomposition can be proved by the induction.
\end{proof}

\indent Lemma \ref{lem1} shows that a tensor $\A\in \SF_{m,n}$ can always be decomposed into the direct sum of irreducible tensors, possibly with a zero block.  The following lemma is dedicated to the necessary and sufficient conditions of $\bcp$ property for irreducible $(0,1)$ tensors.\\

\begin{lem}\label{lem2}
Let $m\ge 2, n\ge 1$ be two integers and $\A\in \SF_{m,n}$ be irreducible.  Then the following statements are equivalent:
\begin{enumerate}
\item[(1)] $\A$ is $\bcp$.
\item[(2)] $\A=\J$ is an all-$1$ tensor.
\item[(3)] The multi-hypergraph $\cP$ associated with tensor $\A$ is a complete block.
\end{enumerate}
\end{lem}

\begin{proof}
If $\A=\J$, then $\A$ is $\bcp$ since $\A=\allone^{m}$ with $\allone=(1,1,\ldots,1)^{\top}$.
Conversely,let $\A\in \SF_{m,n}$ be a $\bcp$ tensor. Then $\A$ has a decomposition (\ref{eq: cp}) with
\[ \bu_{j}=(u_{1j},u_{2j},\ldots, u_{nj})^{\top}\in \Fb^{n}. \]
Then we have
\[a_{i_{1}i_{2}\ldots i_{m}} = \sum\limits_{j=1}^{q} u_{i_{1}j}u_{i_{2}j}\ldots u_{i_{m}j}, ~~ \forall (i_{1},i_{2},\ldots,i_{m})\in S(m,n).  \]
We will first show that $q=1$ in decomposition (\ref{eq:cpu01}).  Suppose that $q>1$.  If there exist a pair of positive integers $(s,t): 1\le s < t \le q$ such that
\[  k\in \supp(\bu_{s})\cap \supp(\bu_{t})  \]
for some $k\in [1\ldots n]$, then $u_{ks}=u_{kt}=1$. Hence we have
\beyy
A_{kk\ldots k} & = & \sum\limits_{j=1}^{q} u_{kj}u_{kj}\ldots u_{kj} \\
                          & = & \sum\limits_{j=1}^{q} u_{kj}^{m}\\
                          &\ge & u_{ks}^{m} + u_{kt}^{m}=2
\eeyy
a contradiction to the assumption that $\A$ is a (0,1) tensor.  Thus we have
\beq\label{eq: disjointsupp}
\supp(\bu_{i})\cap \supp(\bu_{j}) = \emptyset, \forall 1\le i < j\le q
\eeq
Now we define
\[ \D_{i}=\set{\si\in \capEb: B(\si)\subseteq \supp(\bu_{i})}, \forall i=1,2,\ldots,q \]
Then we get $\set{\D_{1},\D_{2}, \ldots,\D_{q}}$ each a subset of $\capEb$,  and
\[ \D_{i}\cap \D_{j}=\emptyset, \quad   \forall 1\le i < j\le q   \]
Denote $V_{i}=V(\D_{i})$ and $\cP_{i}:=(V_{i}, \D_{i})$ for $i=1,2,\ldots,q$.  Then
\[ \cP = \cP_{1}\cup \cP_{2}\cup \ldots \cup \cP_{q} \]
where $\cP=(V, \capEb)$ is the multi-hypergraph associated with $\A$. It turns that
$\A\sim_{p} \A_{1}\oplus \ldots \oplus\A_{q}$ where $\A_{i}$ is the adjacency tensor of $\cP_{i}$, a
contradiction to the hypothesis that $\A$ is irreducible.  Hence $q=1$,  and thus there exists a vector $\bu=(u_{1},\ldots,u_{n})^{\top}\in \Fb^{n}$ such that $\A=\bu^{m}$. \\
\indent To prove that $\A=\J=\allone^{m}$, we need only to show that  $\supp(\bu)=[1\ldots n]$. In fact, if $\supp(\bu)$ is a proper subset of $[1\ldots n]$, then by setting
$I=[1\ldots n]\backslash \supp(\bu)$, we show that $\A$ is reducible by definition, which is a contradiction to the hypothesis. Thus $\supp(\bu)=[1\ldots n]$ and  $\A=\J$. Thus the equivalence between (i) and (ii) is obtained.
\indent The remaining part of the lemma is immediate by definition.
\end{proof}

 \indent From Lemma \ref{lem2} and its proof, we can get the following equivalent conditions for $\bcp$ tensors.
\begin{thm}\label{thm1}
Let $m\ge 2, n\ge 1$ be two positive integers. Suppose that $\A\in \SF_{m,n}$ have no zero blocks and is associated with multi-hypergraph $\cP=(V,\capEb)$.  Then the following are equivalent:
\begin{description}
\item[(1)] $\A$ is $\bcp$\  tensor.
\item[(2)] $\cP$ can be decomposed as the union of some complete blocks $\cP_{i}$ of size $n_{i}$ where $n_{1}+\ldots +n_{q}=n$.
\item[(3)] $\A$ can be written in form (\ref{eq:cpu01}) and with $\bu_j\in\Fb^n$ satisfying  $U^{T}U=diag(n_{1},\ldots,n_{q})$ where $U=[\bu_{1},\ldots,\bu_{q}]$.
\end{description}
\end{thm}

\begin{proof} To prove $(1) \Leftrightarrow (2)$, we first let $\A\in \SF_{m,n}$ be a $\bcp$
tensor.  Then by Lemma \ref{lem1} $\A$ can be written in form (\ref{eq:irredufact}) where each
$\A_{i}$ is an irreducible $\bcp$ tensor of $m$th order $n_{i}$-dimension (no zero block there
since $\A$ has no zero block).  By Lemma \ref{lem2}, $\A_{i}$ is associated with a multi-hypergraph $\cP_{i}=(V_{i}, \capEb_{i})$ where $\abs{V_{i}}=n_{i}$ for $i=1,2,\ldots, q$, $n_{1}+n_{2}+\ldots +n_{q}=n$.  For each $i\in [q]$, by Lemma \ref{lem2},  $\cP_{i}$ is
the complete block of  dimension $n_{i}$ (since $\A_{i}$ is irreducible and  $\bcp$ ). Thus $(1)\Rightarrow (2)$ is proved.  The proof of  $(2)\Rightarrow (1)$ is immediate if we
note that the decomposition (\ref{eq:cpu01}) holds by take $\supp(\bu_{i}) =V_{i}$ for $i=1,2,\ldots, q$. \\
\indent Now we show $(1)\Leftrightarrow (3)$.  First we assume that $\A\in \SF_{m,n}$ is $\bcp$.  Then from the proof of Lemma \ref{lem2} there exist some vectors
$\bu_{j}\in \Fb^{n}$ such that (\ref{eq: cp})  holds, and
\beq\label{eq: disjointsupp}
\supp(\bu_{i})\cap \supp(\bu_{j}) = \emptyset, \forall 1\le i < j\le q
\eeq
It follows that $U^{T}U=diag(n_{1},\ldots, n_{q})$ for $U=[\bu_{1},\ldots,\bu_{q}]$, where $n_{i}$ is the positive integer described above. Thus $(1)\Rightarrow (3)$ is proved.
The other direction can be proved by reversing the above arguments.
\end{proof}

\section*{Acknowledgement}
This research was supported by the Hong Kong Research Grant Council (No. PolyU 501212, 501913, 15302114 and 15300715). The work was partially done during the first two authors' 
visit at the Hong Kong Polytechnic University in August of 2016.

\end{document}